\documentclass{amsart}
\usepackage{amsmath}
\usepackage{amssymb}
\usepackage{amsfonts}

\newtheorem{theorem}{Theorem}
\theoremstyle{plain}

\newtheorem{corollary}{Corollary}

\newtheorem{definition}{Definition}

\newtheorem{lemma}{Lemma}

\newtheorem{proposition}{Proposition}
\newtheorem{remark}{Remark}

\numberwithin{equation}{section}
\input{tcilatex}

\begin{document}
\title[New general integral inequality]{New general integral inequality for
convex functions and applications }
\author{Mehmet Zeki Sar\i kaya$^{\star \clubsuit }$}
\address{$^{\clubsuit }$Department of Mathematics,Faculty of Science and
Arts, D\"{u}zce University, D\"{u}zce, Turkey}
\email{sarikayamz@gmail.com}
\thanks{$^{\star }$corresponding author}
\author{Hasan Ogunmez$^{\blacklozenge }$}
\address{$^{\blacklozenge }$Department of Mathematics, \ Faculty of Science
and Arts, Afyon Kocatepe University, Afyon-TURKEY}
\email{hogunmez@aku.edu.tr}
\author{Mustafa Kemal Y\i ld\i z}
\address{Department of Mathematics, \ Faculty of Science and Arts, Afyon
Kocatepe University, Afyon-TURKEY}
\email{myildiz@aku.edu.tr}
\date{}

\begin{abstract}
In this paper, we establish new general inequality for convex functions.
Then, we apply this inequality to obtain the midpoint, trapezoid and
averaged midpoint-trapezoid integral inequality. Also, some applications for
special means of real numbers are provided.
\end{abstract}

\maketitle

\section{Introduction}

Let $f:I\subseteq \mathbb{R\rightarrow R}$ be a convex mapping defined on
the interval $I$ of real numbers and $a,b\in I$, with $a<b.$ the following
double inequality is well known in the literature as the Hermite-Hadamard
inequality:%
\begin{equation*}
f\left( \frac{a+b}{2}\right) \leq \frac{1}{b-a}\int_{a}^{b}f\left( x\right)
dx\leq \frac{f\left( a\right) +f\left( b\right) }{2}.
\end{equation*}

Let $f:[a,b]\mathbb{\rightarrow R}$ be a differentiable mapping on $(a,b)$
whose derivative $f^{^{\prime }}:(a,b)\mathbb{\rightarrow R}$ is bounded on $%
(a,b),$ i.e., $\left\Vert f^{\prime }\right\Vert _{\infty }=\underset{t\in
(a,b)}{\overset{}{\sup }}\left\vert f^{\prime }(t)\right\vert <\infty .$
Then the following inequality%
\begin{equation*}
\left\vert f(x)-\frac{1}{b-a}\int\limits_{a}^{b}f(t)dt\right\vert \leq \left[
\frac{1}{4}+\frac{(x-\frac{a+b}{2})^{2}}{(b-a)^{2}}\right] (b-a)\left\Vert
f^{\prime }\right\Vert _{\infty }
\end{equation*}%
holds. This result is known in the literature as the Ostrowski inequality%
\cite{Ostrowski}.

In the realm of real functions of real variable, convex functions constitute
a conspicuous body both because they are frequently encountered in practical
applications, and because they satisfy a number of useful inequalities and
theorems [see, \cite{Cerone}-\cite{Dragomir}, \cite{Milovanovic}]. The most
important of the inequalities is of course the defining one which states
that a real function $f(x)$ defined on a real-numbers interval $I=[a,b]$ is
convex if, for any three elements $x_{1},x,x_{2}$ of $I$ and $x_{1}<x_{2}$
such that $x_{1}\leq x\leq x_{2}$,

\begin{equation*}
f(x)\leq f(x_{1})[(x_{2}-x)/(x_{2}-x_{1})]+f(x_{2})[(x-x_{1})/(x_{2}-x_{1})].
\end{equation*}%
Graphically, this means that the point $\{x,f(x)\}$ never falls above the
straight line segment connecting the points $\{x_{1},f(x_{1})\}$ and $%
\{x_{2},f(x_{2})\}$.

\begin{definition}[\protect\cite{Ujevic1}]
Let $f:[a,b]\rightarrow \mathbb{R}$ be a given function. We say that $f$ is
an even function with respect to the point $t_{0}=\frac{a+b}{2}$ if $%
f(a+b-t)=f(t)$ for $t\in \lbrack a,b].$ We say that $f$ is an odd function
with respect to the point $t_{0}=\frac{a+b}{2}$ if $f(a+b-t)=-f(t)$ for $%
t\in \lbrack a,b].$
\end{definition}

Here, we use the term even(odd) function for a given $f:[a,b]\rightarrow 
\mathbb{R}$ if $f$ is even(odd) with respect to the point $t_{0}=\frac{a+b}{2%
}$. We know that each function $f:[a,b]\rightarrow \mathbb{R}$ can be
represented as a sum of one even and one odd function,%
\begin{equation*}
f(t)=f_{1}(t)+f_{2}(t)
\end{equation*}%
where 
\begin{equation*}
f_{1}(t)=\frac{f(t)+f(a+b-t)}{2}
\end{equation*}%
is an even function and 
\begin{equation*}
f_{2}(t)=\frac{f(t)-f(a+b-t)}{2}
\end{equation*}%
is an odd function.

It is not difficult to verify the following facts:

i) If $f$ is an odd function, then $\left\vert f\right\vert $ is an even
function.

ii) If $f,g$ are even or odd functions, then $fg$ is an even function.

iii) If $f$ is an even function and $g$ is an odd function, then $fg$ is an
odd function.

iv) If $f$ is an integrable and odd function, then $\dint%
\limits_{a}^{b}f(x)dx=0.$ Indeed, we have%
\begin{equation*}
\dint\limits_{a}^{b}f(x)dx=-\dint\limits_{a}^{b}f(a+b-x)dx=-\dint%
\limits_{a}^{b}f(u)du=-\dint\limits_{a}^{b}f(x)dx.
\end{equation*}%
Thus, $2\dint\limits_{a}^{b}f(x)dx=0$ and we proved the above assertion.

v) If $f$ is an integrable and even function, then 
\begin{equation*}
\dint\limits_{a}^{b}f(x)dx=2\dint\limits_{\frac{a+b}{2}}^{b}f(x)dx=2\dint%
\limits_{a}^{\frac{a+b}{2}}f(x)dx.
\end{equation*}%
Indeed, we have%
\begin{equation*}
\dint\limits_{a}^{b}f(x)dx=\dint\limits_{a}^{\frac{a+b}{2}%
}f(x)dx+\dint\limits_{\frac{a+b}{2}}^{b}f(x)dx,
\end{equation*}%
and since $f$ is an even function,%
\begin{equation*}
\dint\limits_{a}^{\frac{a+b}{2}}f(x)dx=\dint\limits_{a}^{\frac{a+b}{2}%
}f(a+b-x)dx=-\dint\limits_{b}^{\frac{a+b}{2}}f(u)du=\dint\limits_{\frac{a+b}{%
2}}^{b}f(x)dx.
\end{equation*}%
Thus, the above assertion holds.

In this article, our work is motivated by the works of N. Ujevic \cite%
{Ujevic1} and Z. Liu \cite{Liu}. We obtain new general integral inequaliy
for convex functions. Finally, new error bounds for the midpiont, trapezoid
and other are obtained. Some applications for special means of real numbers
are also provided.

\section{Main Results}

In order to prove our main results, we need the following identity:

\begin{lemma}
\label{lm} Let $f:I\subset \mathbb{R}\rightarrow \mathbb{R}$ be twice
differentiable function on $I^{\circ }$ with $f^{\prime \prime }\in
L_{1}[a,b]$, then%
\begin{align}
& 2\dint_{a}^{b}f(t)dt-\left( \beta -\alpha \right) \left[ f(x)+f(a+b-x)%
\right]   \notag \\
&  \notag \\
& +\left( b-\beta \right) ^{2}f^{\prime }(b)-\left( a-\alpha \right)
^{2}f^{\prime }(a)+2\left( a-\alpha \right) f(a)-2\left( b-\beta \right) f(b)
\notag \\
&  \label{2} \\
& +\left( \beta -\alpha \right) \left[ \left( x-\frac{3\alpha +\beta }{4}%
\right) f^{\prime }(x)+\left( a+b-x-\frac{\alpha +3\beta }{4}\right)
f^{\prime }(a+b-x)\right]   \notag \\
&  \notag \\
& =\dint_{a}^{b}k\left( a,b,t\right) f^{\prime \prime }(t)dt  \notag
\end{align}%
where%
\begin{equation}
k(a,b,t):=\left\{ 
\begin{array}{ll}
\left( t-\alpha \right) ^{2} & ,a\leq t<x \\ 
&  \\ 
\left( t-\frac{\alpha +\beta }{2}\right) ^{2} & x\leq t<a+b-x \\ 
&  \\ 
\left( t-\beta \right) ^{2} & a+b-x\leq t\leq b%
\end{array}%
\right. \text{with }a\leq \alpha <\beta \leq b  \label{20}
\end{equation}%
for any $x\in \lbrack a,\frac{a+b}{2}].$
\end{lemma}

\begin{proof}
It suffices to note that%
\begin{eqnarray*}
I &=&\dint_{a}^{b}k\left( a,b,t\right) f^{\prime \prime }(t)dt \\
&& \\
&=&\dint_{a}^{x}\left( t-\alpha \right) ^{2}f^{\prime \prime
}(t)dt+\dint_{x}^{a+b-x}\left( t-\frac{\alpha +\beta }{2}\right)
^{2}f^{\prime \prime }(t)dt+\dint_{a+b-x}^{b}\left( t-\beta \right)
^{2}f^{\prime \prime }(t)dt \\
&& \\
&=&I_{1}+I_{2}+I_{3}.
\end{eqnarray*}%
By inegration by parts, we have the following identity%
\begin{eqnarray*}
I_{1} &=&\dint_{a}^{x}\left( t-\alpha \right) ^{2}f^{\prime \prime }(t)dt \\
&& \\
&=&\left( x-\alpha \right) ^{2}f^{\prime }(x)-\left( a-\alpha \right)
^{2}f^{\prime }(a)+2\left( a-\alpha \right) f(a)-2\left( x-\alpha \right)
f(x)+2\dint_{a}^{x}f(t)dt.
\end{eqnarray*}%
Similarly, we observe that%
\begin{eqnarray*}
I_{2} &=&\dint_{x}^{a+b-x}\left( t-\frac{\alpha +\beta }{2}\right)
^{2}f^{\prime \prime }(t)dt \\
&& \\
&=&\left( a+b-x-\frac{\alpha +\beta }{2}\right) ^{2}f^{\prime
}(a+b-x)-\left( x-\frac{\alpha +\beta }{2}\right) ^{2}f^{\prime }(x) \\
&& \\
&&+2\left( x-\frac{\alpha +\beta }{2}\right) f(x)-2\left( a+b-x-\frac{\alpha
+\beta }{2}\right) f(a+b-x)+2\dint_{x}^{a+b-x}f(t)dt
\end{eqnarray*}%
and%
\begin{eqnarray*}
I_{3} &=&\dint_{a+b-x}^{b}\left( t-\beta \right) ^{2}f^{\prime \prime }(t)dt
\\
&& \\
&=&\left( b-\beta \right) ^{2}f^{\prime }(b)-\left( a+b-x-\beta \right)
^{2}f^{\prime }(a+b-x) \\
&& \\
&&+2\left( a+b-x-\beta \right) f(a+b-x)-2\left( b-\beta \right)
f(b)+2\dint_{a+b-x}^{b}f(t)dt.
\end{eqnarray*}%
Thus, we can write%
\begin{eqnarray*}
I &=&I_{1}+I_{2}+I_{3} \\
&& \\
&=&\left( \beta -\alpha \right) \left[ \left( x-\frac{3\alpha +\beta }{4}%
\right) f^{\prime }(x)+\left( a+b-x-\frac{\alpha +3\beta }{4}\right)
f^{\prime }(a+b-x)\right]  \\
&& \\
&&-\left( \beta -\alpha \right) \left[ f(x)+f(a+b-x)\right] +\left( b-\beta
\right) ^{2}f^{\prime }(b)-\left( a-\alpha \right) ^{2}f^{\prime }(a) \\
&& \\
&&+2\left( a-\alpha \right) f(a)-2\left( b-\beta \right)
f(b)+2\dint_{a}^{b}f(t)dt
\end{eqnarray*}%
which gives the required identity (\ref{2}).
\end{proof}

\begin{corollary}
\label{cr} Under the assumptions Lemma $\ref{lm}$ with $\alpha =a,\ \beta =b,
$ we have the following identity:%
\begin{multline*}
2\dint_{a}^{b}f(t)dt-\left( b-a\right) \left[ f(x)+f(a+b-x)\right] +\left(
b-a\right) \left( x-\frac{3a+b}{4}\right) \left[ f^{\prime }(x)-f^{\prime
}(a+b-x)\right]  \\
\\
=\dint_{a}^{b}k_{1}\left( a,b,t\right) f^{\prime \prime }(t)dt
\end{multline*}%
where%
\begin{equation*}
k_{1}(a,b,t)=\left\{ 
\begin{array}{ll}
\left( t-a\right) ^{2}, & a\leq t<x \\ 
&  \\ 
\left( t-\frac{a+b}{2}\right) ^{2}, & x\leq t<a+b-x \\ 
&  \\ 
\left( t-b\right) ^{2}, & a+b-x\leq t\leq b%
\end{array}%
\right. 
\end{equation*}%
for any $x\in \lbrack a,\frac{a+b}{2}].$
\end{corollary}

The proof of the Corrollary \ref{cr} is proved by Liu in \cite{Liu}. Hence,
our results in Lemma \ref{lm} are generalizations of the corresponding
results of Liu \cite{Liu}.

\begin{corollary}
Under the assumptions Lemma $\ref{lm}$ with $\alpha =\beta =\frac{a+b}{2},$
we have the following identity:%
\begin{equation*}
2\dint_{a}^{b}f(t)dt+\frac{\left( b-a\right) ^{2}}{4}\left[ f^{\prime
}(b)-f^{\prime }(a)\right] -\left( b-a\right) \left[ f(a)+f(b)\right]
=\dint_{a}^{b}\left( t-\frac{a+b}{2}\right) ^{2}f^{\prime \prime }(t)dt.
\end{equation*}
\end{corollary}

Let us show that the kernal $k(a,b,t)$ defined by (\ref{20}) is an even
function if $\ \alpha +\beta =a+b$. Indeed, for $t\in \lbrack a,x)$ we have 
\begin{equation*}
k(a,b,a+b-t)=(a+b-t-\beta )^{2}=(t-\alpha )^{2}=k(a,b,t).
\end{equation*}%
For $t\in \lbrack x,a+b-x)$ we have 
\begin{equation*}
k(a,b,a+b-t)=\left( a+b-t-\frac{\alpha +\beta }{2}\right) ^{2}=\left( t-%
\frac{\alpha +\beta }{2}\right) ^{2}=k(a,b,t).
\end{equation*}%
For $t\in \lbrack a+b-x,b]$ we have 
\begin{equation*}
k(a,b,a+b-t)=(a+b-t-\alpha )^{2}=(t-\beta )^{2}=k(a,b,t).
\end{equation*}%
Hence, $k(a,b,t)$ is an even function.

Now, by using the above lemma, we prove our main theorems:

\begin{theorem}
\label{thm1} Let $f:I\subset \mathbb{R}\rightarrow \mathbb{R}$ be twice
differentiable function on $I^{\circ }$ such that $f^{\prime \prime }\in
L_{1}[a,b]$ where $a,b\in I,$ $a<b$. If $f^{\prime }$ is a convex on $[a,b]$
and $f^{\prime \prime }(t)\geq 0,\ t\in \left[ a,b\right] ,$ then the
following inequality holds:%
\begin{align}
& \left\vert 2\dint_{a}^{b}f(t)dt-\left( \beta -\alpha \right) \left[
f(x)+f(a+b-x)\right] \right.   \notag \\
&  \notag \\
& +\left( b-\beta \right) ^{2}f^{\prime }(b)-\left( a-\alpha \right)
^{2}f^{\prime }(a)+2\left( a-\alpha \right) f(a)-2\left( b-\beta \right) f(b)
\notag \\
&  \label{3} \\
& \left. +\left( \beta -\alpha \right) \left[ \left( x-\frac{3\alpha +\beta 
}{4}\right) f^{\prime }(x)+\left( a+b-x-\frac{\alpha +3\beta }{4}\right)
f^{\prime }(a+b-x)\right] \right\vert   \notag \\
&  \notag \\
& \leq \left\Vert k\right\Vert _{\infty }\left[ f^{\prime }(b)-f^{\prime }(a)%
\right] ,\ \text{for any }x\in \lbrack a,\frac{a+b}{2}]  \notag
\end{align}%
where $\left\Vert k\right\Vert _{\infty }=\underset{t\in \left[ a,b\right] }{%
\max }\left\vert k(a,b,t)\right\vert .$
\end{theorem}

\begin{proof}
From Lemma \ref{lm}$,$ we get,%
\begin{eqnarray}
&&\left\vert 2\dint_{a}^{b}f(t)dt-\left( \beta -\alpha \right) \left[
f(x)+f(a+b-x)\right] \right.   \notag \\
&&  \notag \\
&&+\left( b-\beta \right) ^{2}f^{\prime }(b)-\left( a-\alpha \right)
^{2}f^{\prime }(a)+2\left( a-\alpha \right) f(a)-2\left( b-\beta \right) f(b)
\notag \\
&&  \notag \\
&&\left. +\left( \beta -\alpha \right) \left[ \left( x-\frac{3\alpha +\beta 
}{4}\right) f^{\prime }(x)+\left( a+b-x-\frac{\alpha +3\beta }{4}\right)
f^{\prime }(a+b-x)\right] \right\vert   \label{4} \\
&&  \notag \\
&\leq &\dint_{a}^{b}\left\vert k\left( a,b,t\right) \right\vert \left\vert
f^{\prime \prime }(t)\right\vert dt.  \notag
\end{eqnarray}%
Let us consider the following notations%
\begin{equation*}
f_{1}^{\prime \prime }(t)=\frac{f^{\prime \prime }(t)+f^{\prime \prime
}(a+b-t)}{2},\ \ \ f_{2}^{\prime \prime }(t)=\frac{f^{\prime \prime
}(t)-f^{\prime \prime }(a+b-t)}{2},
\end{equation*}%
then we have $f^{\prime \prime }(t)=f_{1}^{\prime \prime }(t)+f_{2}^{\prime
\prime }(t)$ and $k\left( a,b,t\right) f_{2}^{\prime \prime }(t)$ is an odd
function while $\left\vert f_{1}^{\prime \prime }(t)\right\vert $ and $%
k\left( a,b,t\right) f_{1}^{\prime \prime }(t)$ are even functions. Thus, by
using properties (iv) we obtain%
\begin{eqnarray*}
\dint_{a}^{b}k\left( a,b,t\right) f^{\prime \prime }(t)dt
&=&\dint_{a}^{b}k\left( a,b,t\right) \left[ f_{1}^{\prime \prime
}(t)+f_{2}^{\prime \prime }(t)\right] dt \\
&& \\
&=&\dint_{a}^{b}k\left( a,b,t\right) f_{1}^{\prime \prime }(t)dt.
\end{eqnarray*}%
Thus by using properties (v) we get%
\begin{eqnarray}
\left\vert \dint_{a}^{b}k\left( a,b,t\right) f^{\prime \prime
}(t)dt\right\vert  &=&\left\vert \dint_{a}^{b}k\left( a,b,t\right)
f_{1}^{\prime \prime }(t)dt\right\vert   \notag \\
&&  \notag \\
&\leq &\dint_{a}^{b}\left\vert k\left( a,b,t\right) \right\vert \left\vert
f_{1}^{\prime \prime }(t)\right\vert dt  \notag \\
&&  \notag \\
&\leq &\left\Vert k\right\Vert _{\infty }\dint_{a}^{b}\left\vert
f_{1}^{\prime \prime }(t)\right\vert dt  \notag \\
&&  \notag \\
&=&2\left\Vert k\right\Vert _{\infty }\dint_{\frac{a+b}{2}}^{b}\left\vert
f_{1}^{\prime \prime }(t)\right\vert dt  \label{5} \\
&&  \notag \\
&=&\left\Vert k\right\Vert _{\infty }\dint_{\frac{a+b}{2}}^{b}\left\vert
f^{\prime \prime }(t)+f^{\prime \prime }(a+b-t)\right\vert dt.  \notag
\end{eqnarray}%
Therefore, since $f^{\prime \prime }(t)\geq 0,\ t\in \left[ a,b\right] ,\ $%
we get%
\begin{eqnarray}
\dint_{\frac{a+b}{2}}^{b}\left\vert f^{\prime \prime }(t)+f^{\prime \prime
}(a+b-t)\right\vert dt &=&\dint_{\frac{a+b}{2}}^{b}\left[ f^{\prime \prime
}(t)+f^{\prime \prime }(a+b-t)\right] dt  \label{51} \\
&&  \notag \\
&=&f^{\prime }(b)-f^{\prime }(a).  \notag
\end{eqnarray}%
Using (\ref{5}) and (\ref{51}) in (\ref{4}), we obtain (\ref{3}) which
completes the proof.
\end{proof}

\begin{theorem}
\label{thm2} Let $f:I\subset \mathbb{R}\rightarrow \mathbb{R}$ be twice
differentiable function on $I^{\circ }$ such that $f^{\prime \prime }\in
L_{1}[a,b]$ where $a,b\in I,$ $a<b$. If $\left\vert f^{\prime \prime
}\right\vert $ is a convex on $[a,b],$ then the following inequality holds:%
\begin{align}
& \left\vert 2\dint_{a}^{b}f(t)dt-\left( \beta -\alpha \right) \left[
f(x)+f(a+b-x)\right] \right.   \notag \\
&  \notag \\
& +\left( b-\beta \right) ^{2}f^{\prime }(b)-\left( a-\alpha \right)
^{2}f^{\prime }(a)+2\left( a-\alpha \right) f(a)-2\left( b-\beta \right) f(b)
\notag \\
&  \label{8} \\
& \left. +\left( \beta -\alpha \right) \left[ \left( x-\frac{3\alpha +\beta 
}{4}\right) f^{\prime }(x)+\left( a+b-x-\frac{\alpha +3\beta }{4}\right)
f^{\prime }(a+b-x)\right] \right\vert   \notag \\
&  \notag \\
& \leq \frac{b-a}{4}\left\Vert k\right\Vert _{\infty }\left[ \left\vert
f^{\prime \prime }(a)\right\vert +\left\vert f^{\prime \prime
}(b)\right\vert +2\left\vert f^{\prime \prime }(\frac{a+b}{2})\right\vert %
\right] ,\ \text{for any }x\in \lbrack a,\frac{a+b}{2}]  \notag
\end{align}%
where $\left\Vert k\right\Vert _{\infty }=\underset{t\in \left[ a,b\right] }{%
\max }\left\vert k(a,b,t)\right\vert .$
\end{theorem}

\begin{proof}
By similar computation the proof of Theorem $\ref{thm1},$ we get%
\begin{eqnarray}
&&\left\vert 2\dint_{a}^{b}f(t)dt-\left( \beta -\alpha \right) \left[
f(x)+f(a+b-x)\right] \right.   \notag \\
&&  \notag \\
&&+\left( b-\beta \right) ^{2}f^{\prime }(b)-\left( a-\alpha \right)
^{2}f^{\prime }(a)+2\left( a-\alpha \right) f(a)-2\left( b-\beta \right) f(b)
\notag \\
&&  \label{9} \\
&&\left. +\left( \beta -\alpha \right) \left[ \left( x-\frac{3\alpha +\beta 
}{4}\right) f^{\prime }(x)+\left( a+b-x-\frac{\alpha +3\beta }{4}\right)
f^{\prime }(a+b-x)\right] \right\vert   \notag \\
&&  \notag \\
&\leq &\left\Vert k\right\Vert _{\infty }\dint_{\frac{a+b}{2}}^{b}\left[
\left\vert f^{\prime \prime }(t)\right\vert +\left\vert f^{\prime \prime
}(a+b-t)\right\vert \right] dt.  \notag
\end{eqnarray}%
Since $\left\vert f^{\prime \prime }\right\vert $ is a convex on $[a,b],$ by
Hermite-Hadamard's integral inequality we have%
\begin{eqnarray}
&&\dint_{\frac{a+b}{2}}^{b}\left[ \left\vert f^{\prime \prime
}(t)\right\vert +\left\vert f^{\prime \prime }(a+b-t)\right\vert \right] dt 
\notag \\
&&  \notag \\
&=&\dint_{\frac{a+b}{2}}^{b}\left\vert f^{\prime \prime }(t)\right\vert
dt+\dint_{a}^{\frac{a+b}{2}}\left\vert f^{\prime \prime }(t)\right\vert dt
\label{10} \\
&&  \notag \\
&\leq &\frac{b-a}{4}\left[ \left\vert f^{\prime \prime }(a)\right\vert
+\left\vert f^{\prime \prime }(b)\right\vert +2\left\vert f^{\prime \prime }(%
\frac{a+b}{2})\right\vert \right] .  \notag
\end{eqnarray}%
Therefore, using (\ref{10}) in (\ref{9}), we obtain (\ref{8}) which
completes the proof.
\end{proof}

\section{Applications to Quadrature Formulas}

In this section we point out some particular inequalities which generalize
some classical results such as : trapezoid inequality, Ostrowski's
inequality, midpoint inequality and others.

\begin{proposition}
\label{p1} Under the assumptions Theorem $\ref{thm1},$ we have%
\begin{eqnarray}
&&\left\vert \int\limits_{a}^{b}f(t)dt-\frac{b-a}{2}\left[ f(x)+f(a+b-x)%
\right] +\frac{b-a}{2}(x-\frac{3a+b}{4})\left[ f^{\prime }(x)-f^{\prime
}(a+b-x)\right] \right\vert   \notag \\
&&  \label{6} \\
&\leq &\frac{f^{\prime }(b)-f^{\prime }(a)}{3}\left[ \left( x-a\right)
^{3}+\left( \frac{a+b}{2}-x\right) ^{3}\right] ,\ \ \text{for any }x\in
\lbrack a,\frac{a+b}{2}].  \notag
\end{eqnarray}
\end{proposition}

\begin{proof}
If we choose $\alpha =a,\ \beta =b$ in (\ref{20}), then we obtain $%
\left\Vert k\right\Vert _{\infty }=\frac{2}{3}\left[ (x-a)^{3}+(\frac{a+b}{2}%
-x)^{3}\right] .$ Thus, from the inequality (\ref{3}) it follows that (\ref%
{6}) holds.
\end{proof}

\begin{remark}
If we put $x=\frac{a+b}{2}$ in (\ref{6}), we get the "midpoint inequality":%
\begin{equation}
\left\vert \frac{1}{b-a}\int\limits_{a}^{b}f(t)dt-f\left( \frac{a+b}{2}%
\right) \right\vert \leq \frac{(b-a)^{2}}{24}\left[ f^{\prime }(b)-f^{\prime
}(a)\right] .  \label{H1}
\end{equation}
\end{remark}

\begin{proposition}
\label{p2} Under the assumptions Theorem $\ref{thm1},$ we have%
\begin{eqnarray}
&&\left\vert \frac{1}{b-a}\int\limits_{a}^{b}f(t)dt-\frac{f(a)+f(b)}{2}+%
\frac{\left( b-a\right) }{8}\left[ f^{\prime }(b)-f^{\prime }(a)\right]
\right\vert   \notag \\
&&  \label{7} \\
&\leq &\frac{(b-a)^{2}}{48}\left[ f^{\prime }(b)-f^{\prime }(a)\right] . 
\notag
\end{eqnarray}
\end{proposition}

\begin{proof}
If we choose $\alpha =\beta =\frac{a+b}{2}$ in (\ref{20}), then we obtain $%
\left\Vert k\right\Vert _{\infty }=\frac{(b-a)^{3}}{24}.$ Thus, from the
inequality (\ref{3}) it follows that (\ref{7}) holds.
\end{proof}

Another particular integral inequality with many applications is the
following one:

\begin{proposition}
Under the assumptions Theorem $\ref{thm2},$ we have%
\begin{eqnarray}
&&\left\vert \frac{1}{b-a}\int\limits_{a}^{b}f(t)dt-\frac{f(a)+f(b)}{2}+%
\frac{\left( b-a\right) }{8}\left[ f^{\prime }(b)-f^{\prime }(a)\right]
\right\vert   \notag \\
&&  \label{11} \\
&\leq &\frac{(b-a)^{3}}{192}\left[ \left\vert f^{\prime \prime
}(a)\right\vert +\left\vert f^{\prime \prime }(b)\right\vert +2\left\vert
f^{\prime \prime }\left( \frac{a+b}{2}\right) \right\vert \right] .  \notag
\end{eqnarray}
\end{proposition}

\begin{proof}
If we choose $\alpha =\beta =\frac{a+b}{2}$ in (\ref{20}), then we obtain $%
\left\Vert k\right\Vert _{\infty }=\frac{(b-a)^{3}}{24}.$ Thus, from the
inequality (\ref{8}) it follows that (\ref{11}) holds.
\end{proof}

\begin{remark}
It is clear that the best estimation we can have in (\ref{11}) for $%
f^{\prime }(b)=f^{\prime }(a)$ is getting the "trapezoid inequality":%
\begin{equation}
\left\vert \frac{1}{b-a}\int\limits_{a}^{b}f(t)dt-\frac{f(a)+f(b)}{2}%
\right\vert \leq \frac{(b-a)^{3}}{192}\left[ \left\vert f^{\prime \prime
}(a)\right\vert +\left\vert f^{\prime \prime }(b)\right\vert +2\left\vert
f^{\prime \prime }\left( \frac{a+b}{2}\right) \right\vert \right] .
\label{12}
\end{equation}
\end{remark}

\section{Applications for special means}

Recall the following means:

(a) The arithmetic mean 
\begin{equation*}
A=A(a,b):=\dfrac{a+b}{2},\ a,b\geq 0;
\end{equation*}

(b) The geometric mean 
\begin{equation*}
G=G(a,b):=\sqrt{ab},\text{ }a,b\geq 0;
\end{equation*}

(c) The harmonic mean 
\begin{equation*}
H=H\left( a,b\right) :=\dfrac{2ab}{a+b},\ a,b>0;
\end{equation*}

(d) The logarithmic mean 
\begin{equation*}
L=L\left( a,b\right) :=\left\{ 
\begin{array}{ccc}
a & if & a=b \\ 
&  &  \\ 
\frac{b-a}{\ln b-\ln a} & if & a\neq b%
\end{array}%
\right. \text{, \ \ \ }a,b>0;
\end{equation*}

(e) The identric mean%
\begin{equation*}
I=I(a,b):=\left\{ 
\begin{array}{ccc}
a & if & a=b \\ 
&  &  \\ 
\frac{1}{e}\left( \frac{b^{b}}{a^{a}}\right) ^{\frac{1}{b-a}} & if & a\neq b%
\end{array}%
\right. \text{, \ \ \ }a,b>0;
\end{equation*}

(f) The $p-$logarithmic mean:

\begin{equation*}
L_{p}=L_{p}(a,b):=\left\{ 
\begin{array}{ccc}
\left[ \frac{b^{p+1}-a^{p+1}}{\left( p+1\right) \left( b-a\right) }\right] ^{%
\frac{1}{p}} & \text{if} & a\neq b \\ 
&  &  \\ 
a & \text{if} & a=b%
\end{array}%
\right. \text{, \ \ \ }p\in \mathbb{R\diagdown }\left\{ -1,0\right\} ;\;a,b>0%
\text{.}
\end{equation*}%
It is also known \ that $L_{p}$ is monotonically nondecreasing in $p\in 
\mathbb{R}$ with $L_{-1}:=L$ and $L_{0}:=I.$ The following simple
relationships are known in the literature%
\begin{equation*}
H\leq G\leq L\leq I\leq A.
\end{equation*}%
Now, using the results of Section 3, some new inequalities is derived for
the above means.

\begin{proposition}
Let $p>1$ and $0<a<b.$ Then we have the inequality:%
\begin{equation*}
\left\vert L_{p}^{p}(a,b)-A^{p}\left( a,b\right) \right\vert \leq p\frac{%
\left( b-a\right) ^{2}}{12}\left( b^{p-1}-a^{p-1}\right) .
\end{equation*}
\end{proposition}

\begin{proof}
The assertion follows from (\ref{H1}) applied for $f(t)=t^{p},$ $t\in \left[
a,b\right] .$ We omitted the details.
\end{proof}

\begin{proposition}
Let $p>1$ and $0<a<b.$ Then we have the inequality:%
\begin{equation*}
\left\vert L_{p}^{p}(a,b)-A^{p}\left( a,b\right) +p(p-1)\frac{\left(
b-a\right) ^{2}}{8}L_{p}^{p-2}(a,b)\right\vert \leq p\frac{\left( b-a\right)
^{2}}{24}\left( b^{p-1}-a^{p-1}\right) .
\end{equation*}
\end{proposition}

\begin{proof}
The assertion follows from (\ref{7}) applied for $f(x)=x^{p},$ $x\in \left[
a,b\right] .$
\end{proof}

\begin{proposition}
Let $0<a<b.$ Then we have the inequality:%
\begin{equation*}
\left\vert \ln \left[ I(a,b)G(a,b)\right] +\frac{\left( b-a\right) ^{2}}{8}%
G^{-2}(a,b)\right\vert \leq \frac{\left( b-a\right) ^{3}}{96}\left[
H^{-1}(a^{2},b^{2})+\frac{1}{2}A^{-2}(a,b)\right] .
\end{equation*}
\end{proposition}

\begin{proof}
The assertion follows from (\ref{11}) applied for $f(x)=-\ln t,$ $t\in \left[
a,b\right] .$
\end{proof}

\begin{proposition}
Let $0<a<b.$ Then we have the inequality:%
\begin{equation*}
\left\vert L^{-1}(a,b)-H^{-1}(a,b)\right\vert \leq \frac{\left( b-a\right)
^{3}}{48}\left[ H^{-1}(a^{3},b^{3})+\frac{1}{2}A^{-3}(a,b)\right] .
\end{equation*}
\end{proposition}

\begin{proof}
The assertion follows from (\ref{12}) applied for $f(x)=\frac{1}{t},$ $t\in %
\left[ a,b\right] .$
\end{proof}

\end{document}